\documentclass[11pt, a4paper]{amsart}
\usepackage{amsfonts,amssymb,amsmath,amsthm,amscd,mathtools,multicol,tikz, tikz-cd,caption,enumerate,mathrsfs,thmtools,cite}
\usepackage{inputenc}
\usepackage[foot]{amsaddr}
\usepackage[pagebackref=true, colorlinks, linkcolor=blue, citecolor=red]{hyperref}
\usepackage{latexsym}
\usepackage{fullpage}
\usepackage{microtype}
\usepackage{subfiles} 

\usepackage{palatino}
\makeatletter
\makeindex
\newcommand{\be}{\begin{equation}}
\newcommand{\ee}{\end{equation}}
\newcommand{\beano}{\begin{eqn*}} 
	\newcommand{\eeano}{\end{eqnarray*}}
\newcommand{\ba}{\begin{array}}
	\newcommand{\ea}{\end{array}}

\declaretheoremstyle[headfont=\normalfont]{normalhead}

\newtheorem{theorem}{Theorem}[section]
\newtheorem{theoremalph}{Theorem}[section]

\newtheorem{lemma}[theorem]{Lemma}
\newtheorem{corollary}[theorem]{Corollary}
\newtheorem{definition}[theorem]{Definition}
\newtheorem{proposition}[theorem]{Proposition}
\newtheorem{remark}[theorem]{Remark}
\newtheorem{example}[theorem]{Example}
\newcommand{\diag}{\mathrm{diag}}

\numberwithin{equation}{section}

\newcommand{\bt}{\begin{theorem}}
\newcommand{\et}{\end{theorem}}
\newcommand{\bco}{\begin{corollary}}
\newcommand{\eco}{\end{corollary}}
\newcommand{\bd}{\begin{definition}}
\newcommand{\ed}{\end{definition}}
\newcommand{\bp}{\begin{problem}}
\newcommand{\ep}{\end{problem}}
\newcommand{\bl}{\begin{lemma}}
\newcommand{\el}{\end{lemma}}
\newcommand{\bprop}{\begin{proposition}}
\newcommand{\eprop}{\end{proposition}}
\newcommand{\br}{\begin{remark}}
\newcommand{\er}{\end{remark}}
\newcommand{\bpf}{\begin{proof}}
\newcommand{\epf}{\end{proof}}
\newcommand{\bex}{\begin{example}}
\newcommand{\eex}{\end{example}}

\begin{document}
\title{Finite Splittings of Differential Matrix Algebras}
\author{Amit Kulshrestha}
\email{amitk@iisermohali.ac.in}
\author{Kanika Singla}
\email{ph15016@iisermohali.ac.in}
\address{IISER Mohali, Knowledge City, Sector 81, Mohali 140 306 INDIA}
\thanks{This work is supported by DST, India through the research grant EMR/2016/001516.}

\begin{abstract}
It is well known that central simple algebras are split by suitable finite Galois extensions of their centers. In \cite{Magid-Lourdes-DCSA}
a counterpart of this result was studied in the set up of differential matrix algebras, wherein Picard-Vessiot extensions that split
matrix differential algebras were constructed. In this article, we exhibit instances of differential matrix algebras which are split by finite extensions. In some cases, we relate the existence of finite splitting extensions of a differential matrix algebra to the triviality of its tensor powers, and show in these cases, that orders of differential matrix algebras divide their degrees.
\end{abstract}
\maketitle

\section{Introduction} Let $(K,\delta)$ be a differential field of characteristic zero.
Let $A$ be a central simple algebra over $K$ and $\mathcal D : A \to A$ be a derivation
such that $\mathcal D|_K = \delta$. Such a pair $(A,\mathcal{D})$ is called
a \emph{differential central simple algebra} over $(K, \delta)$.
If $A$ is a matrix algebra over $K$ and $\mathcal D$ is the derivation $\delta^c$ obtained by applying $\delta$ to each matrix entry,
then $(A,\mathcal D)$ is called a \emph{split} differential central simple algebra.
A differential extension $(E, \delta_E)$ of $(K, \delta)$ is said to be a \emph{splitting field} of $(A, \mathcal D)$ if the differential central simple algebra $(A, \mathcal D) \otimes (E, \delta_E)$ 
obtained after a base change to $(E, \delta_E)$ is split.
In \cite{Magid-Lourdes-DCSA}, Juan and Magid prove that if the field of constants of $\delta$ is algebraically closed,
then every differential central simple algebra $(A, \mathcal D)$ admits a Picard-Vessiot extension $(E,\delta_E)$ that is also a splitting field of $(A, \mathcal D)$. Picard-Vessiot extensions have a large transcendence degree. However,
in this article, we are interested in finite splitting fields of differential matrix algebras. 
A splitting extension $(E, \delta_E)$ of $(K, \delta)$ is called \emph{finite} if the
degree of $E$ over $K$ is finite.

If $\mathcal D$ is a derivation on a matrix algebra $M_n(K)$ such that $\mathcal D|_K = \delta$, then by an application of Skolem-Noether theorem, $\mathcal D(X) = \delta^c(X) + PX - XP$, where $P$ is a traceless matrix. We prove three types of results in this article.

\begin{enumerate}
\item Conditions on $P$ that ascertain the existence of finite splitting extensions. (Theorem \ref{split-by-finite})
\item Calculation of the minimal transcendence degree of splitting extensions, if finite splitting extensions
do not exist. (see \S\ref{section-transcendental-extensions-splitting-dcsa}) 
\item Relating the existence of finite splitting extensions to the existence of an $m \in \mathbb N$ such that
the $m$-fold tensor power $(M_n(K), \mathcal D)^{\otimes m}$ is split. (Theorem \ref{no-algebraic-then-infinite-order}, Corollary \ref{order-divides-degree})
\end{enumerate}
Since a central simple algebra is split by a finite extension, the case of differential matrix algebras is
a significant step while dealing with splitting fields of arbitrary differential central simple algebras.

The organization of the article is as follows. In  \S\ref{section-preliminaries}, we set up notations, briefly recall the results of \cite{Magid-Lourdes-DCSA} and formulate lemmas on which our proofs are based.
We develop tools for the proofs of the main results of the article in \S \ref{splitting-fields-of-S}. 
Given a finite subset $S \subseteq K \setminus \{0\}$, we find differential extensions, called \emph{splitting fields of} $S$, that contain solutions of $\delta(y) = ay$ for every $a \in S$. The crucial part of this section is the decomposition
$\mathcal A_S + \mathcal A_S' = \langle S \rangle$, where
$\mathcal A$ is the $\mathbb Q$-span of the set of logarithmic derivatives of $\delta$, $\langle S \rangle$ is the additive subgroup of $K$ generated by $S$,
$\mathcal A_S = \mathcal A \cap \langle S \rangle$ and
$\mathcal A_S'$ is a $\mathbb Z$-complement of $\mathcal A_S$ in
$\langle S \rangle$. The group $\mathcal A_S$ governs the
algebraic part, and the group $\mathcal A_S'$ governs the
transcendental part of a splitting field of $S$. The following theorem makes it more precise.

\begin{theoremalph} \label{TheoremA}
Let $S \subseteq K \backslash\{0\}$ be a finite subset.
If $\mathcal A_S' = 0$, then there exists a splitting field $(E, \delta_E)$ of $S$ such that the field extension $E/K$ is finite. If $\mathcal A_S' \neq 0$ and $(E, \delta_E)$ is a splitting field of $S$, then $E$ is transcendental over $K$ and ${\rm trdeg}_K(E) \geq {\rm rank}(\mathcal A_S')$. Further, this lower bound on ${\rm trdeg}_K(E)$ is sharp.
(Theorem \ref{transcendental-iff-AS-prime-nonzero}, Corollary \ref{trdeg-and-rank-AS-prime}, Proposition \ref{deg_tr})
\end{theoremalph}

In the last two sections of this article, we fix the differential field to be $(K, \delta)$, where
$K = F(t)$ and $\delta$ is the derivation on $K$ determined by $\delta(t) = t$ and $\delta(a) = 0$ for $a \in F$.
Here, $F$ is a field of characteristic $0$.
The main section of the paper is \S\ref{section-algebraic-extensions-splitting-dcsa}, where we prove the following theorem concerning the existence of a finite splitting field of matrix differential algebras over $(K, \delta)$.

\begin{theoremalph}\label{TheoremB}
Let $P \in M_n(K)$ be a traceless matrix and $\mathcal D_P(X) = \delta^c(X) + PX - XP$.
\begin{enumerate} \setlength\itemsep{0.5em}
\item[(i).] If $P \in M_n(F)$ then $(M_n(K) , \mathcal{D}_P)$ admits a finite splitting field if and only if $P$ is diagonalizable and every eigenvalue of $P$ is rational. (Theorem \ref{split-by-finite})
\item[(ii).] 
If $P \notin M_n(F)$ then $(M_n(K), \mathcal{D}_P)$ admits a finite splitting field if $P$ has one of the following forms. (Theorem \ref{splitting-of-non-constant-matrices})
\begin{enumerate} 
\item $P={\rm diag}(d_1,d_2,\dots,d_n)$ and each
$d_i \in \mathcal A$.
\item $P = (p_{ij})$, $p_{ij}\in F[t]$ is an upper triangular matrix satisfying 
$p_{ii}\in \mathbb Q$ and $p_{ii}\mathbb Z\neq p_{jj}\mathbb Z\in \mathbb Q/\mathbb Z$, unless $i = j$.
\end{enumerate}
\end{enumerate}
\end{theoremalph}

Subsequently, in \S\ref{section-transcendental-extensions-splitting-dcsa}, over a more general derivation on $K = F(t)$ defined by $\delta_{c,m}(a) = 0$ for $a \in F$ and $\delta_{c,m}(t) = ct^m$, $c \in F \setminus\{0\}$, $m \in \mathbb Z$, we compute the minimum transcendence degree of a splitting field of $(M_n(K), \mathcal D_P)$, where $P \in M_n(F)$. We notice that the only case when finite splitting extensions exist for such derivations is $m = 1$.

Finally, in \S\ref{section-splitting-of-tensor-powers}, we relate the existence of finite splitting fields to the splitting of a tensor power of differential matrix algebras over the differential field $(K,\delta)$.
Smallest integer $r$ such that $(M_n(K), \mathcal D_P)^{\otimes r}$ is split is called the \emph{order}
of $(M_n(K), \mathcal D_P)$.

\begin{theoremalph}
Let $P \in M_n(F)$ be a traceless matrix.
If the order of $(M_n(K),\mathcal{D}_P)$ is finite, then
$(M_n(K),\mathcal{D}_P)$ admits a finite splitting field over $(K, \delta)$.
Further, the order is a divisor of $n$. 
(Theorem \ref{no-algebraic-then-infinite-order}, Corollary \ref{order-divides-degree})
\end{theoremalph}

\noindent{\bfseries Acknowledgement.}~ We thank Varadharaj R. Srinivasan for helpful discussions and his interest in this work. We also thank B. Sury for valuable comments on a draft of this paper.

\section{Preliminaries}\label{section-preliminaries}
In this section, we recall relevant definitions and record some lemmas to be used subsequently in the paper.
\subsection{Basic definitions}
Throughout the paper, $K$ will denote a field of characteristic zero. A specific instance of $K$ will be explicitly mentioned, wherever needed.
Let $A$ be an associative algebra over $K$. An additive map $\mathcal D : A \to A$ is called a \emph{derivation} if 
$\mathcal D(ab)= \mathcal D(a)b+a\mathcal D(b)$. 
The field $K$ together with a derivation $\delta : K \to K$ is called a \emph{differential field}.
For a differential field $(K,\delta)$, the collection
$C_K=\{a\in K : \delta(a)=0\}$ forms a subfield of $K$ and is called the \emph{field of constants} of $(K,\delta)$.
A derivation of the form $\partial_u$ for some $u \in A$, where 
$\partial_u(x)=ux-xu$,  for all 
$x \in A$, is called an \emph{inner derivation} of $A$.
If $(A,\mathcal D)$ and $(K, \delta)$ are such that $\delta = \mathcal D|_K$ then $(A,\mathcal D)$ is called a \emph{differential algebra}
over $(K,\delta)$ and $\mathcal D$ is called an extension of $\delta$.
\\

If $E\supseteq K$ is a field extension with a derivation $\delta_E$ such that $\delta_E\big|_K=\delta$ then $(E,\delta_E)$ is called an \emph{extension of the differential field} $(K, \delta)$ or a \emph{differential extension} of $(K, \delta)$. A differential extension $(E,\delta_E)$ of $(K, \delta)$ is called \emph{algebraic} (respectively, \emph{transcendental} or \emph{finite}) if the field extension $E/K$ is so.
If $(K, \delta)$ is a differential field and $E/K$ is an algebraic extension of fields, then a derivation $\delta$ of $K$ extends uniquely to a derivation of $E$.	
A differential extension $(E, \delta_E)$ over $(K, \delta)$ is called a \emph{Picard-Vessiot extension} if
it does not admit new constants, \emph{i.e.}, $C_E = C_K$, and $E$ can be differentially generated by adjoining the solutions of a homogeneous linear ordinary differential equation to $K$. \\

If $A$ is a central simple $K$-algebra, then every derivation $\delta$ of $K$ extends to a derivation $\mathcal D$ of $A$.
(see \cite[Theorem A]{Amitsur}). The pair $(A,\mathcal{D})$, where $A$ is a central simple 
algebra over $K$ and $\mathcal{D}$ is a derivation on $A$ that extends a derivation $\delta$ of $K$, is called a 
\emph{differential central simple algebra} over $(K, \delta)$. 
If $(A_1, \mathcal{D}_1)$ and $(A_2, \mathcal{D}_2)$ are two differential central simple algebras over a differential field
$(K, \delta)$ then we define their tensor product by
$(A_1, \mathcal{D}_1) \otimes (A_2, \mathcal{D}_2) = (A_1 \otimes_K A_2, \mathcal{D}_1 \otimes \mathcal{D}_2) $,
where $ \mathcal{D}_1 \otimes \mathcal{D}_2(a_1\otimes a_2) =\mathcal{D}_1(a_1)\otimes a_2+a_1 \otimes \mathcal{D}_2(a_2).$\\

If $\mathcal D_1$ and $\mathcal D_2$ are two derivations on a central simple algebra $A$ which extend $\delta$, then
by an application of \emph{Skolem-Noether Theorem}
$\mathcal{D}_1-\mathcal{D}_2$ is inner (see \cite[Theorem 4.9]{Jacobson-II}), and hence there exists $u\in A$ such that $\mathcal{D}_1-\mathcal{D}_2=\partial_u$. 
Consequently, for a matrix algebra $M_n(K)$, every derivation $\mathcal{D}$ that extends $\delta$ is of the form $\mathcal{D}_P := \delta^c + \partial_P$, where
$P$ is a traceless matrix.\\

\subsection{Splitting of differential central simple algebras}
Let $(A, \mathcal D)$ be a differential 
central simple algebra of degree $n$ over $(K,\delta)$. If there exists a differential extension $(E, \delta_E)$ over $(K, \delta)$
such that $(A, \mathcal{D}) \otimes (E, \delta_E)$ is differentially isomorphic to $(M_n(E),\delta_E^c)$, where $\delta_E^c(B)=(\delta_E(b_{ij}))$ for every $B=(b_{ij})\in M_n(E)$, then 
$(A,\mathcal D)$ is said to be \emph{split by $(E, \delta_E)$} or \emph{trivialized by $(E,\delta_E)$}. 
If $(A, \mathcal{D})$ is isomorphic to $(M_n(K), \delta^c)$, 
then $(A, \mathcal{D})$ is called a split differential central simple algebra over $(K,\delta)$.

We now briefly recall the results of Juan and Magid \cite{Magid-Lourdes-DCSA}.
\begin{proposition}\label{isomorphism}\cite[Proposition 2]{Magid-Lourdes-DCSA}
Let $(K, \delta)$ be a differential field and $P, Q \in M_n(K)$ be traceless matrices. Differential matrix algebras 
$(M_n(K), \mathcal{D}_{P})$ and $(M_n(K), \mathcal{D}_{Q})$ 
are isomorphic if and only if there exists $H \in {\rm GL}_n(K)$ such that $H^{-1}\delta^c(H)+H^{-1}QH=P$. 
In particular, a differential matrix algebra $(M_n(K),\mathcal{D}_P)$ is split if and only if $P = H^{-1} \delta^c(H)$
for some $H\in {\rm GL}_n(K)$.
\end{proposition}

\begin{corollary}
If $\delta$ is the zero derivation on $K$, then $(M_n(K), \mathcal{D}_P)$ is isomorphic to $(M_n(K), \mathcal{D}_Q)$ if and only if 
$P$ and $Q$ are similar matrices.
\end{corollary}

\begin{corollary}\label{trivialize}
Let $(E, \delta_E)$ be a differential extension of $(K, \delta)$. Then $(M_n(K), \mathcal{D}_P)$ is split by $(E, \delta_E)$ if and only if ${\rm GL}_n(E)$ contains a solution of the equation $\delta^c(Y)=P^t Y$.
\end{corollary}

Later in this paper, we will frequently use the above corollary in conjunction with the following lemma.

\begin{lemma}\label{similar-matrices}
Let $(K, \delta)$ be a differential field, and $C_K$ be its field of constants. 
Let $A,B\in M_n(K)$, and $Q\in M_n(C_K)$ be such that $A=Q^{-1}BQ$. Let $(E, \delta_E) \supseteq (K, \delta)$ be an extension of
differential fields.
Then, $\delta^c(Y)=A^tY$ has a solution for $Y$ in ${\rm GL}_n(E)$ if and only if $\delta^c(Y)=B^tY$ has a solution for $Y$ in ${\rm GL}_n(E)$.
\end{lemma}
\begin{proof}
Suppose that $\delta^c(Y)=A^tY$ has a solution $Z \in {\rm GL}_n(E)$. Then
$$\delta_E^c(Z) = A^tZ = Q^tB^t(Q^{-1})^tZ = Q^tB^t(Q^{-1})^tZ Q^t (Q^{-1})^t.$$
Let $Z_1=(Q^{-1})^t Z Q^t$. Since $\delta^c(Q)=0$, we have
$$\delta_E^c(Z_1)=\delta_E^c((Q^{-1})^t Z Q^t)=(Q^{-1})^t \delta_E^c(Z) Q^t = B^t(Q^{-1})^t Z Q^t = B^tZ_1.$$
Thus a solution of $\delta^c(Y)=B^tY$ is $Z_1=(Q^{-1})^tZQ^t\in {\rm GL}_n(E)$. 
Converse holds because the hypothesis of the lemma is symmetric in $A$ and $B$.
\end{proof}

\subsection{Order of $(A,\mathcal D)$}
A differential central simple algebra $(A, \mathcal D)$ over $(K, \delta)$ is said to be of \emph{finite order} if there is an integer $m \geq 1$ such that
$(A, \mathcal D)^{\otimes m} := \underbrace{(A, \mathcal D) \otimes (A, \mathcal D)\otimes \dots \otimes (A, \mathcal D)}_{m \,\, times}$ is 
split over $(K, \delta)$ itself. If $(A, \mathcal D)$ is of finite order, then the smallest
$m$ for which $(A, \mathcal D)^{\otimes m}$ is split is called the \emph{order of $(A, \mathcal D)$}.

The lemmas recorded in this subsection would be referred to while proving results concerning orders of differential
matrix algebras.
\begin{lemma}\label{tensor}
Let $(A_1,\mathcal D_{P_1})$ and $(A_2,\mathcal D_{P_2})$ be two differential matrix algebras over $(K, \delta)$. 
Then $\mathcal{D}_{P_1} \otimes \mathcal{D}_{P_2}=\mathcal{D}_{P_1 \otimes 1 + 1 \otimes P_2}$, 
where $P_1\in A_1, P_2 \in A_2$ respectively.
\end{lemma}
\begin{proof}
For $X_1 \in A_1$ and $X_2 \in A_2$,
\begin{align*}
(\mathcal{D}_{P_1} \otimes \mathcal{D}_{P_2})(X_1 \otimes X_2)&=(\mathcal{D}_{P_1} \otimes 1+ 1 \otimes \mathcal{D}_{P_2})(X_1 \otimes X_2)
\\&=(\mathcal{D}_{P_1} \otimes 1)(X_1 \otimes X_2) + (1 \otimes \mathcal{D}_{P_2}) (X_1 \otimes X_2) 
\\&=(\delta^c(X_1)+P_1X_1-X_1P_1)\otimes X_2+ X_1 \otimes (\delta^c(X_2)+P_2X_2-X_2P_2) \\
&=\delta^c(X_1) \otimes X_2 + X_1 \otimes \delta^c(X_2) + P_1X_1 \otimes X_2 + X_1 \otimes P_2X_2 \\&\quad- X_1P_1 \otimes X_2 -X_1 \otimes X_2P_2
\\&= \delta^c(X_1 \otimes X_2) + (P_1 \otimes 1 + 1 \otimes P_2)(X_1 \otimes X_2)\\&\quad- (X_1 \otimes X_2)(P_1 \otimes 1 + 1 \otimes P_2)
\\&= \mathcal{D}_{P_1 \otimes 1 + 1 \otimes P_2}(X_1 \otimes X_2).
\end{align*}
\end{proof}

If $\mathcal D_Q = \mathcal D_{P_1} \otimes \mathcal D_{P_2}$ then 
$Q=P_1 \otimes 1 + 1 \otimes P_2$, where $\otimes$ is the Kronecker product of matrices. 
By \cite[Theorem 4.4.5]{Topics-in-matrix-analysis} eigenvalues of $Q$ are the sums of the eigenvalues of $P_1$ and $P_2$, 
\emph{i.e.},
if $\Lambda_1$ and $\Lambda_2$ denote the sets of eigenvalues of $P_1$ and $P_2$, respectively, then
the set of eigenvalues of $Q$ is $\Lambda=\{\lambda_1 + \lambda_2 : \lambda_1 \in \Lambda_1, \lambda_2 \in \Lambda_2\}.$

\begin{lemma}\label{tensor_n}
Let $(A_i,\mathcal D_{P_i})$, $i\in \{1,2,\dots,n\}$ be differential matrix algebras over $(K, \delta)$. 
Let $$(A,\mathcal{D}_Q)=(A_1, \mathcal D_{P_1}) \otimes (A_2, \mathcal D_{P_2})\otimes \dots \otimes (A_n, \mathcal D_{P_n}).$$ 
For each $i\in \{1,2,\dots,n\}$, let $\Lambda_i$ be the set of eigenvalues of $P_i$ and
$\Lambda$ be the set of eigenvalues of $Q$. Then

\begin{enumerate}[(i)]
\item $\Lambda=\{\lambda_1+\lambda_2+\dots+\lambda_n : \lambda_i \in \Lambda_i\}$.
In particular, the matrix $Q$ has all its eigenvalues in $\mathbb{{Q}}$ if and only if each $P_i$ has all its eigenvalues in $\mathbb{{Q}}$.
\item $Q$ is diagonalizable if and only if each $P_i$ is diagonalizable.
\end{enumerate}
\end{lemma}

\begin{proof}
$(i)$.~ Since the Kronecker product of matrices is associative (see \cite[Property 4.2.6]{Topics-in-matrix-analysis}),
 $\Lambda=\{\lambda_1+\lambda_2+\dots+\lambda_n : \lambda_i \in \Lambda_i\}$ holds inductively, with 
$n = 2$ as the base case that is discussed immediately before this lemma. 
Therefore, if each $\Lambda_i \subset \mathbb Q$ then $\Lambda \subset \mathbb Q$.
Conversely, suppose that $\Lambda \subset \mathbb Q$.
Then for a given $\lambda \in \Lambda_i$ and arbitrary $\lambda_j \in \Lambda_j$ we have
\begin{align} \label{one-eigen-value}
\lambda + \sum_{j \neq i}\lambda_{j} \in \mathbb Q.
\end{align}

Since each $P_j$ is a traceless matrix, $\sum_{\lambda_j \in \Lambda_j} {\lambda_j} = 0$. Using this after adding up all instances
of (\ref{one-eigen-value}) as $\lambda_j$ varies over $\Lambda_j$, we have 
$\left(\prod_{j \neq i}|\Lambda_j|\right) \lambda \in \mathbb Q$, where $|\Lambda_j|$ is the
cardinality of the set $\Lambda_j$. Clearly, $\lambda \in \mathbb Q$. \\

$(ii)$.~ As the tensor product
$(A_1, \mathcal D_{P_1}) \otimes (A_2, \mathcal D_{P_2})\otimes \dots \otimes (A_n, \mathcal D_{P_n})$
can be constructed repeatedly by taking one tensor factor each time,
it suffices to demonstrate the proof for $n = 2$ case. \\

Suppose $n = 2$. By Lemma \ref{tensor}, $Q = P_1\otimes 1+1\otimes P_2$. Let $J_{P_1}$ and $J_{P_2}$ be Jordan forms of 
$P_1$ and $P_2$, respectively. Let $R_1$ and $R_2$ be such that $R_1^{-1}P_1R_1 = J_{P_1}$ and $R_2^{-1}P_2R_2 = J_{P_2}$. Note that
$R_1 \otimes R_2$ is invertible and $Q=(R_1 \otimes R_2)(J_{P_1}\otimes 1+1\otimes J_{P_2})(R_1 \otimes R_2)^{-1}$.
Thus if $P_1$ and $P_2$ are diagonalizable, the Jordan forms $J_{P_1}$ and $J_{P_2}$ are diagonal matrices, and so is 
$J_{P_1}\otimes 1+1\otimes J_{P_2}$. In other words, $Q$ is diagonalizable. \\

Conversely, if at least one of the matrices, say $P_2$, is not diagonalizable. Then there exists an eigenvalue $\beta$ of $P_2$ whose algebraic multiplicity is strictly greater than its geometric multiplicity. We may assume, by permuting Jordan blocks of $J_{P_2}$ if needed, that the first Jordan block of $J_{P_2}$ corresponds to $\beta$. Let
$\alpha$ be the eigenvalue corresponding to the first Jordan block of $J_{P_1}$. We claim that the algebraic multiplicity of the eigenvalue
$\alpha + \beta$ of $Q := P_1 \otimes 1 + 1 \otimes P_2$ is strictly greater than its geometric multiplicity. To see this, we look at the conjugate
$Q' := J_{P_1} \otimes 1 + 1 \otimes J_{P_2} = (a_{ij})_{r \times r}$
of $Q$ and observe that $Q'$ is an upper triangular matrix with $a_{11} = a_{22} = \alpha + \beta$, $a_{12} = 1$. Let $s$ denote the algebraic multiplicity of $\alpha + \beta$. To estimate its geometric multiplicity, we find a lower bound on the rank of $Q' - (\alpha+ \beta)I_r$, where $I_r$ is the $r \times r$ identity matrix.
In $Q' - (\alpha+ \beta)I_r$ the first row and $r-s$ rows corresponding to eigenvalues different from $\alpha + \beta$ constitute a linearly independent set of size $r-s+1$. Thus, ${\rm rank}(Q' - (\alpha+ \beta)I_r) \geq r-s +1$.
By rank-nullity theorem, the geometric multiplicity of $\alpha + \beta$ is at most $r-(r-s+1) = s-1 < s$. Thus $Q$ is not  diagonalizable.
\end{proof}

As a consequence, we derive the following lemma that will be used later while dealing with tensor powers of differential matrix algebras.

\begin{lemma}\label{m_fold}
Let $(K, \delta)$ be a differential field and $P \in M_n(K)$ be a traceless matrix. Let $Q$ be a traceless matrix such that
$(M_{n^m}(K),\mathcal{D}_Q)=(M_n(K), \mathcal D_P)^{\otimes m}$. 
Then the following hold for $Q$.
\begin{enumerate}[(i)]
\item $Q$ is diagonalizable if and only if $P$ is diagonalizable.
\item $Q$ has all its eigenvalues in $\mathbb{{Q}}$ if and only if $P$ has all its eigenvalues in $\mathbb{{Q}}$.
\end{enumerate}
\end{lemma}

\section{Splitting fields of $S \subseteq K\backslash\{0\}$} \label{splitting-fields-of-S}
Let $(K, \delta)$ be a differential field.
Consider the set $$\mathcal A = \{a \in K : \delta(y) = nay \text{ has a nonzero solution in } K \text{ for some } n \in \mathbb N\}.$$ 
Then $\mathcal A$ is an additive subgroup of $K$ that is closed under division by natural numbers.
Let $a \in \mathcal A$. Let $m,n\in \mathbb N$ and
$u,v \in K$ be such that $\delta(u)=mau$ and $\delta(v)=na v$. Then $\delta(uv^{-1})=(m-n)a uv^{-1}$. 
Thus, for each $a \in \mathcal A$, there exists a unique positive integer $n_a$ with the following property: $\delta(y)=nay$ has a solution in $K\backslash\{0\}$ if and only if $n$ is an integral multiple of $n_a$. 

An element $a \in K$ is called a \emph{logarithmic derivative} if $a = \delta(b)b^{-1}$ for some $b \in K \backslash\{0\}$. It is evident that if $a$ is a logarithmic derivative, then $a \in \mathcal A$. We denote the set of logarithmic derivatives by
$\mathcal L_{\delta}(K)$.
The following lemma provides a description of $\mathcal A$ in terms of $\mathcal L_{\delta}(K)$.

\begin{lemma}
The subgroup $\mathcal A$ of $K$ is equal to the $\mathbb Q$-span of $\mathcal L_{\delta}(K)$.
\end{lemma}

\begin{proof}
Since $\mathcal L_{\delta}(K)$ is a subset of $\mathcal A$ and $\mathcal A$ is an additive subgroup of $K$ which is closed under division by natural numbers,
${\rm span}_{\mathbb Q}(\mathcal L_{\delta}(K)) \subseteq \mathcal A$.

Conversely, let $a\in \mathcal A$. Then there exists
$v_a \in K \setminus \{0\}$ and $n_a \in \mathbb N$ such that $\delta(v_a)=n_a av_a$. 
Hence $a = (1/n_a)\delta(v_a)v_a^{-1} \in {\rm span}_{\mathbb Q}(\mathcal L_{\delta}(K))$.
\end{proof}

Let $S \subseteq K \backslash\{0\}$.
We are interested in extensions $(E, \delta_E)$ of $(K, \delta)$ which contain a solution of $\delta(y) = a y$ for every $a \in S$. Such extensions are to be called {\it splitting fields} of $S$. Let $a, b \in S$ and $(E, \delta_E)$ be an extension of $(K, \delta)$ such that
$\delta_E(u_a) = a u_a$ and $\delta_E(u_b) = b u_b$ for some $u_a, u_b \in E$. Then $u_a u_b^{\pm1} \in E$ is a solution of $\delta(y) = (a \pm b)y$. From this, we derive the following lemma.
In the lemma, and thereafter, $\langle S \rangle$ denotes the additive subgroup of $K$ generated by $S$.

\begin{lemma}\label{splitting-field-same}
Let $(E, \delta_E)$ be a differential extension of $(K, \delta)$. Then $(E, \delta_E)$ is a splitting field of $S$ if and only if it is a splitting field of $\langle S \rangle$.
\end{lemma}

Let $S \subseteq K \backslash\{0\}$ be a finite subset.
Since ${\rm char}(K) = 0$, the group $\langle S \rangle$ and its subgroups are finitely generated free abelian groups. Let us denote $\mathcal A \cap \langle S \rangle$ by $\mathcal A_S$. 
Let $\mathcal A_S'$ denote a $\mathbb Z$-complement of $\mathcal A_S$ in $\langle S \rangle$ so that 
the sum $\mathcal A_S + \mathcal A_S' = \langle S \rangle$ is direct. 

\begin{theorem}\label{transcendental-iff-AS-prime-nonzero}
Let $S \subseteq K \backslash\{0\}$ be a finite subset.
\begin{enumerate}
\item[(i).] If $\mathcal A_S' = 0$, then there exists a splitting field $(E, \delta_E)$ of $S$ such that the field extension $E/K$ is finite.
\item[(ii).] If $\mathcal A_S' \neq 0$ and $(E, \delta_E)$ is splitting field of $S$, then $E$ is transcendental over $K$.
\end{enumerate}
\end{theorem}

\begin{proof}
$(i).$ Since $\mathcal A_S' = 0$, $S \subseteq \mathcal A$. Let $a \in S \subseteq \mathcal A$. By definition of $\mathcal A$, there exist $v_a \in K \setminus \{0\}$ and $n \in \mathbb N$ such that $\delta(v_a) = nav_a$. Let $E_a$ be a field extension of $K$ that contains a root of the polynomial $x^n - v_a \in K[x]$. Then a solution of $x^n - v_a = 0$ in $E_a$ is also a solution of 
$\delta(y) = ay$ in $E_a$. Varying $a$ over $S$ and repeating this procedure, one may construct a finite extension $E$ of $K$ as required.

$(ii).$ Assume that $\mathcal A_S' \neq 0$. It is enough to exhibit an element $a \in \langle S \rangle$ such that no algebraic extension of $(K, \delta)$ contains a solution of $\delta(y) = ay$.
We claim that any nonzero element of $\mathcal A_S'$  serves the purpose.
Let $0 \neq a \in \mathcal A_S'$. Let $(E, \delta_E)$ be an extension of $(K, \delta)$ and $u_a \in E$ be such that $\delta_E(u_a) = au_a$. If $u_a$ is algebraic over $K$ then we take the minimal monic polynomial
$f(x)=x^n+a_{n-1}x^{n-1}+\dots +a_1x+a_0$; $a_0 \neq 0$
of $u_a$ over $K$.
Differentiating the relation $$u_a^n+a_{n-1}u_a^{n-1}+\dots+a_1u_a+a_0=0,$$
we get
$$(nu_a^{n-1}+(n-1)a_{n-1}u_a^{n-2}+\dots+a_0)a u_a+\delta(a_{n-1})u_a^{n-1}+\dots+\delta(a_0)=0.$$
Thus the polynomial
$g(x) := nax^n + ((n-1)a_{n-1}a + \delta(a_{n-1}))x^{n-1} + \dots + \delta(a_0)$
of degree $n$ is satisfied by $u_a$. Therefore, $g(x)= naf(x)$. Now comparing constant terms of these polynomials, $a =(1/n)\delta(a_0)a_0^{-1} \in {\rm span}_{\mathbb Q}(\mathcal L_{\delta}(K)) = \mathcal A$.
Thus $a \in (\mathcal A \cap \langle S \rangle) \cap \mathcal A_S' = 0$. This is a contradiction. 
\end{proof}

Now we shall compute the minimal transcendence degree over $K$ amongst the splitting fields of $S$. If $\mathcal A_S' = 0$; the minimal transcendence degree is $0$, as is evident from Theorem \ref{transcendental-iff-AS-prime-nonzero}$(i)$. Thus we would be interested in $\mathcal A_S' \neq 0$ case. For this we introduce some notation.

Let $S \subseteq K \backslash\{0\}$ be a finite subset with $\mathcal A_S' \neq 0$ and $(E, \delta_E)$ be an extension of $(K, \delta)$ that contains elements $s_a \in E$ such that $\delta_E(s_a) = a s_a$ for $a \in \mathcal A_S'$.
Scaling each $s_a \in E$ by a suitable element of $C_E$, if required, we may assume that
$s_a s_b = s_{a+b}$ for each $a, b \in \mathcal A_S'$. We may also assume $s_0 = 1$. 

\begin{proposition}\label{independence-of_s_a}
With the notation as above, we have the following.
\begin{enumerate}
\item[(i).] The subset $\{s_a : a\in \mathcal A_S'\}$ of $E$ is linearly independent over $K$.
\item[(ii).] Let $a_1,a_2, \dots, a_r \in \mathcal A_S'$. Then $\{a_1, a_2, \dots, a_r\}$ is a $\mathbb Q$-linearly independent subset of $K$ if and only if  $s_{a_1}, s_{a_2}, \dots, s_{a_r}\in E$ are algebraically independent over $K$.
\end{enumerate}
\end{proposition}

\begin{proof}
$(i).$ 
Let $r \in \mathbb N$ be the smallest integer such that
\begin{equation}\label{l_d_eqn}
\sum_{i=1}^{r}\ell_i s_{a_i}=0
\end{equation}
for $0 \neq \ell_i \in K$.
Applying $\delta_E$ on both sides of the equation \ref{l_d_eqn}, we get
\begin{equation}\label{l_d_eqn_diff}
(\delta(\ell_1)+\ell_1a_1)s_{a_1}+\sum_{i=2}^{r}\left(\delta(\ell_i)+\ell_ia_i\right)s_{a_i}=0.
\end{equation}
Multiplying equation \ref{l_d_eqn} by $\delta(\ell_1)+\ell_1 a_1$ and equation \ref{l_d_eqn_diff} by $\ell_1$, and subtracting
\begin{equation}\label{l_d_eqn_smaller}
\sum_{i=2}^{r}((\delta(\ell_1)+\ell_1a_1)\ell_i-(\delta(\ell_i)+\ell_ia_i)\ell_1)s_{a_i}=0.
\end{equation}
By minimality of $r$, each coefficient of $s_{a_i}$ in equation \ref{l_d_eqn_smaller} must be $0$. Thus $$(\delta(\ell_1)+\ell_1a_1)\ell_i-(\delta(\ell_i)+\ell_ia_i)\ell_1=0$$ for each $i\in \{2,\dots, r\}$. Rearranging these terms, $a_1-a_i=\delta(\ell_i)\ell_i^{-1}-\delta(\ell_1)\ell_1^{-1}$.
Evidently $a_1 - a_i \in \mathcal A_S'$ and
$\delta(\ell_i)\ell_i^{-1} \in \mathcal A$.
Thus $a_1 - a_i \in \mathcal A_S' \cap \mathcal A = \{0\}$, and hence $a_1=a_i$ for each $i\in \{2,\dots, r\}$. 
This is a contradiction to the minimality of $r$, and we conclude that $\{s_a : a\in \mathcal A_S'\}$ is a linearly independent set over $K$.

$(ii).$
We first assume that $\{a_1, a_2, \dots, a_r\}$ is a $\mathbb Q$-linearly dependent set. Then there exist nonzero $n_i\in \mathbb Z$, such that $\sum_{i}n_i a_i = 0$. Hence
$\prod_i s_{a_i}^{n_i} = s_{\sum_{i}n_i a_i} = s_0 = 1$ is an algebraic relation over $\mathbb Q$ that is satisfied by $s_{a_1}, s_{a_2}, \dots, s_{a_r}$. 
	
Conversely, assume that $\{a_1, a_2, \dots, a_r\}$ is a $\mathbb Q$-linearly independent set. If possible, let 
$$f(x_1, x_2, \dots, x_r) = \sum_{(i_1,i_2,\dots,i_r) \in \mathcal I}\ell_{(i_1,i_2,\dots,i_r)} x_1^{i_1} x_2^{i_2}\dots x_r^{i_r} \in K[x_1,x_2,\dots,x_r]$$ be a polynomial such that $f(s_{a_1}, s_{a_2}, \dots, s_{a_r}) = 0 \in E$. Then
$$\sum_{(i_1,i_2,\dots,i_r) \in \mathcal I}\ell_{(i_1,i_2,\dots,i_r)} s_{a_1}^{i_1} s_{a_2}^{i_2}\dots s_{a_r}^{i_r} = 0 = 
\sum_{(i_1,i_2,\dots,i_r) \in \mathcal I}\ell_{(i_1,i_2,\dots,i_r)} s_{i_1 a_1 + i_2 a_2 + \dots + i_r a_r} = 0 \in E
$$
Since $\{a_1, a_2, \dots, a_r\}$ is a $\mathbb Q$-linearly independent set, $i_1 a_1 + i_2 a_2 + \dots + i_r a_r$ are all distinct as $(i_1,i_2,\dots,i_r)$ varies over $\mathcal I$. Now from part $(i)$ of this proposition, the set $$\{s_{i_1 a_1 + i_2 a_2 + \dots + i_r a_r} : (i_1,i_2,\dots,i_r) \in \mathcal I\}$$ is linear independent over $K$. 
Thus  $\ell_{(i_1,i_2,\dots,i_r)} = 0$ for each $(i_1,i_2,\dots,i_r) \in \mathcal I$ and
$$f(x_1, x_2, \dots, x_r) = 0 \in K[x_1, x_2, \dots, x_r].$$ Hence $s_{a_1}, s_{a_2}, \dots, s_{a_r}$ are algebraically independent over $K$.
\end{proof}

As a corollary to Proposition \ref{independence-of_s_a}
we derive the following.

\begin{corollary}\label{trdeg-and-rank-AS-prime}
Let $(K, \delta)$ be a differential field and $S\subseteq K\backslash\{0\}$
be a finite set. Let $(E, \delta_E)$ be a splitting field of $S$. Then ${\rm trdeg}_K(E) \geq {\rm rank}(\mathcal A_S')$.
\end{corollary}

\begin{proof}
Since $(E, \delta_E)$ is a splitting field of $S$, by Lemma
\ref{splitting-field-same} it is also a splitting field of 
$\langle S \rangle$. In particular, the equation $\delta(y) = ay$ has a solution $s_a \in E$ for each $a \in \mathcal A_S'$. Let $\mathcal B_S'$ be a $\mathbb Z$-basis of 
$\mathcal A_S'$. By Proposition \ref{independence-of_s_a},
the elements $s_a$, where $a \in \mathcal B_S'$, are algebraically independent over $K$. Thus ${\rm trdeg}_K(E) \geq |\mathcal B_S'| = {\rm rank}(\mathcal A_S')$.
\end{proof}

We assert that there indeed exists a splitting field $(E, \delta_E)$ of $S$ such that ${\rm trdeg}_K(E) = {\rm rank}(\mathcal A_S')$. To establish this assertion, we construct a differential extension $(K(t_S), \delta_{K(t_S)})$ of $(K, \delta)$ as follows: Fix a $\mathbb Z$-basis $\mathcal B_S'$ of $\mathcal A_S'$. Let $t_S := \{t_a : a \in \mathcal B_S'\}$ be a set, in bijection with $\mathcal B_S'$, consisting of algebraically independent mutually commuting indeterminates over $K$. We denote by $K(t_S)$ the field generated by $K$ and $t_S$. We extend the derivation $\delta$ of $K$ to $K(t_S)$ by defining $\delta_{K(t_S)}(t_a) = at_a$ for each $a \in \mathcal B_S'$. It is easy to check that 
$$\delta_{K(t_S)}(t_a^m t_b^n)=(ma+nb)t_a^m t_b^n,$$
for $m,n \in \mathbb Z$ and $a, b \in \mathcal B_S'$.
Therefore $(K(t_S), \delta_{K(t_S)})$ has a solution of the equation $\delta(y) = ay$ for each $a \in \mathcal A_S'$ and the transcendence degree of $K(t_S)$ over $K$ is equal to the rank of $\mathcal A_S'$.

\begin{proposition}\label{deg_tr}
Let $(K, \delta)$ be a differential field and $S \subseteq K\backslash \{0\}$. Then there exists a splitting field $(E, \delta_E)$ of $S$ such that ${\rm trdeg}_K(E) = {\rm rank}(\mathcal A_S')$.
\end{proposition}

\begin{proof}
We claim that there exists an algebraic extension of the field $K(t_S)$ containing solutions of $\delta(y) = ay$ for all $a \in S$.
To see this, let $a \in S$. Then $a = \alpha + \beta$ for suitable $\alpha \in \mathcal A_S$ and $\beta \in \mathcal A_S'$. Fix $\mathbb Z$-bases $\mathcal B_S$ of $\mathcal A_S$ and $\mathcal B_S'$ of $\mathcal A_S'$.
Let $\alpha = \sum_{i=1}^{r} k_i \alpha_i$ and $\beta = \sum_{j=1}^{s} \ell_j \beta_j$, where $\alpha_i \in \mathcal B_S$, $\beta_j\in \mathcal B_S'$ and $k_i, \ell_j \in \mathbb Z$. Since $\alpha_i \in \mathcal B_S \subset \mathcal A$,
there exist $n_i \in \mathbb N$ and $v_i \in K$ such that 
$\delta(v_i) = n_i \alpha_i v_i$. Consider the field
$E = K(t_S)(v_1^{1/n_1}, v_2^{1/n_2}, \dots, v_r^{1/n_r})$. 
The derivation $\delta_{K(t_S)}$ extends uniquely to a 
derivation $\delta_E$ of $E$, and

\begin{equation*}
\begin{split}
\delta_E\left(\prod_{i=1}^r v_i^{k_i/n_i}\right) &=
\sum_{i = 1}^r \delta_E(v_i^{k_i/n_i})\prod_{j \neq i} v_{j}^{k_{j}/n_j} \\
&= \sum_{i = 1}^r (k_i/n_i)  v_i^{k_i/n_i} v_i^{-1} \delta_E(v_i)\prod_{j \neq i} v_{j}^{k_{j}/n_j} \\
&= \sum_{i = 1}^r (k_i/n_i)  v_i^{k_i/n_i} n_i \alpha_i \prod_{j \neq i} v_{j}^{k_{j}/n_j} \\
&= \sum_{i = 1}^r k_i \alpha_i \prod_{j = 1}^r v_j^{k_{j}/n_j} = \alpha \prod_{i = 1}^r v_{i}^{k_{i}/n_i} 
\end{split}
\end{equation*}
and a similar calculation yields

\begin{equation*}
\begin{split}
\delta_E\left(\prod_{j=1}^{s}t_{\beta_j}^{\ell_j}\right) = \beta \prod_{j=1}^{s}t_{\beta_j}^{\ell_j}.
\end{split}
\end{equation*}

Now,
\begin{equation*}
\begin{split}
\delta_E\left(\prod_{i=1}^r v_i^{k_i/n_i}\prod_{j=1}^{s}t_{\beta_j}^{\ell_j}\right) &=  \delta_E\left(\prod_{i=1}^r v_i^{k_i/n_i}\right)\prod_{j=1}^{s}t_{\beta_j}^{\ell_j}+\prod_{i=1}^r v_i^{k_i/n_i}\delta_E\left(\prod_{j=1}^{s}t_{\beta_j}^{\ell_j}\right)\\ &=\alpha \prod_{i=1}^r v_i^{k_i/n_i}\prod_{j=1}^{s}t_{\beta_j}^{\ell_j}+\prod_{i=1}^r v_i^{k_i/n_i} \beta \prod_{j=1}^{s}t_{\beta_j}^{\ell_j}\\
&= (\alpha + \beta) \prod_{i=1}^r v_i^{k_i/n_i}\prod_{j=1}^{s}t_{\beta_j}^{\ell_j} = a \prod_{i=1}^r v_i^{k_i/n_i}\prod_{j=1}^{s}t_{\beta_j}^{\ell_j}.
\end{split}
\end{equation*}
Thus for each $a \in S$, the equation $\delta(y) = ay$ has a solution
$\prod_{i=1}^r v_i^{k_i/n_i}\prod_{j=1}^{s}t_{\beta_j}^{\ell_j} \in E$. Clearly $E$ is an algebraic extension of $K(t_S)$ and
${\rm trdeg}_K(E) = {\rm trdeg}_K(K(t_S)) = {\rm rank}(\mathcal A_S')$.
\end{proof}

We note that the degree of the finite extension $E$ over $K(t_S)$ is at most $n_1 n_2 \dots n_r$. We may choose
$n_i = n_{a_i}$, where $n_{a_i}$ is the smallest positive integer such that $\delta(y) = n_{a_i}a_iy$ has a solution in $K$, to obtain a better bound on the degree of $E$ over $K(t_S)$. 

Thus far, we have constructed a differential field extension of $(K,\delta)$ which contains solutions of differential equations of the form $\delta(y)=ay$ where $a\in S$, and $S\subseteq K \backslash \{0\}$ is a finite set. We shall now focus on solving the system of equations
$\delta(y_1)=ay_1, \delta(y_i)=y_{i-1}+ay_i, \text{for } i \geq 2$.

\begin{lemma}\label{derivative_1}
Let $a \in K$ and $(L, \delta_L)$ be a differential extension of $(K, \delta)$ that contains a solution $u_a$ of $\delta(y) = ay$. Let $n \geq 2$ be an integer.
Consider the following system of equations over $L$.
\begin{equation}\label{jordan-block-solution}
y_1 = u_a, \delta_L(y_i)=y_{i-1}+ay_i, \text{ for } 2 \leq i \leq n
\end{equation}
Then the following are equivalent.
\begin{enumerate}
\item[(i).] There exists $w \in L$ such that $\delta_L(w) = 1$.
\item[(ii).] The differential field $(L, \delta_L)$ contains a solution of \ref{jordan-block-solution}.
\item[(iii).] An algebraic extension of $(L, \delta_L)$ contains a solution of \ref{jordan-block-solution}.
\end{enumerate}
Further, if these equivalent conditions do not hold, then there exists a differential extension $(E,\delta_E)$ over $(L, \delta_L)$ such that ${\rm trdeg}_L(E) = 1$ and $E$ contains a solution of \ref{jordan-block-solution}.
\end{lemma}
\begin{proof}
Suppose $w\in L$ is such that $\delta_L(w)=1$. The equation corresponding to $i=2$ in \ref{jordan-block-solution} is $\delta_L(y_2)=u_a+ay_2$. This can be rewritten as 
$\delta_L(y_2 u_a^{-1}) = 1$. Since $w \in L$ is such that $\delta_L(w) = 1$, there exists $c_2 \in C_L$ such that
$y_2 u_a^{-1}=w+c_2$. Thus $y_2 = (w+c_2)u_a \in L$ is a solution of $\delta_L(y_2)=u_a+ay_2$. Now for $i = 3$, the equation is $\delta_L(y_3)=y_2+ay_3$. Substituting $y_2 = (w+c_2)u_a \in L$, the equation is $\delta_L(y_3)=(w+c_2)u_a  +ay_3$. Again, this can be rewritten as $\delta_L(y_3 u_a^{-1}) = w+c_2$. Since $w+c_2 = \delta_L(w^2/2+c_2w)$, there exists $c_3 \in C_L$ such that $y_3 u_a^{-1} = w^2/2+c_2w+c_3$. Thus
$y_3 = (w^2/2+c_2w+c_3)u_a \in L$ is a solution of $\delta_L(y_3)=y_2+ay_3$. Proceeding this way,

$$y_i = \left(\sum_{j = 0}^{i-1} c_{i-j} (j!)^{-1} w^j 
\right)u_a \in L,$$
where $c_1 = 1$ and $c_{i-j}$ are suitable constants in $L$, is a solution set of the system of equations \ref{jordan-block-solution}.

Now, suppose $(M,\delta_M)$ is an algebraic extension of $(L, \delta_L)$ that contains a solution of \ref{jordan-block-solution}. In particular, there exists $z_2 \in M$ such that
$\delta_M(z_2)=u_a+az_2$. Let $u = z_2 u_a^{-1} \in M$. Then $\delta_M(u) = 1$. If $u \in L$ then $w = u$ satisfies $\delta_L(w) = 1$. Thus we may assume that $u \notin L$. Since $M$ is an algebraic extension of $L$, we consider the
minimal monic polynomial
$f(x)=x^k+a_{k-1} x^{k-1}+\dots + a_1x + a_0$ of $w$ over $L$.
Thus,
$$f(u)=u^k+a_{k-1} u^{k-1}+\dots + a_1 u + a_0=0.$$
Applying $\delta_M$ on both sides of this equation,
$$\delta_M(f(u))=(k+\delta_L(a_{k-1}))u^{k-1}+\dots+(2a_2 + \delta_L(a_1))u+a_1+\delta_L(a_0)=0.
$$

If $k+\delta_L(a_{k-1}) \neq 0$, then we arrive at a contradiction to the minimality of $f(x)$, unless $k = 1$.
However, $k = 1$ too, leads to a contradiction since $u \notin L$. Thus $k+\delta_L(a_{k-1}) = 0$ and $w = -k^{-1}a_{k-1} \in L$ satisfies $\delta_L(w) =1$. 

This establishes the equivalence of three conditions.

If these equivalent conditions do not hold then we extend the derivation $\delta_L$ to the field
$E = L(t)$, by defining $\delta_E(t) = 1$. Since $w = t \in E$ satisfies $\delta_E(w) = 1$, by the above equivalent conditions, the system of equations \ref{jordan-block-solution} has a solution
in $(E, \delta_E)$.
\end{proof}
\begin{corollary}\label{derivative_1_in_alg_ext}
Let $(K,\delta)$ be a differential field and $(E,\delta_E)$ be an algebraic extension of $(K,\delta)$. Then there exists $w\in E$ such that $\delta_E(w)=1$ if and only if there exists $v\in K$ such that $\delta(v)=1$.
\end{corollary}

\begin{lemma}\label{A_K-and-A_K(w)}
Let $(K, \delta)$ be a differential field such that $\delta(z)\neq 1$ for every $z \in K$. Let $L = K(w)$ be a simple transcendental extension of $K$. Let $\delta_L$ be an extension of $\delta$ to $L$ such that
$\delta_L(w) = 1$. Let 
$$\mathcal A_K = \{a \in K : \delta(y) = nay \text{ has a nonzero solution in } K \text{ for some } n \in \mathbb N\}, and$$
$$\mathcal A_L = \{a \in L : \delta(y) = nay \text{ has a nonzero solution in } L \text{ for some } n \in \mathbb N\}.$$
Then $\mathcal A_K=\mathcal A_L\cap K$.
\end{lemma}
\begin{proof}
Clearly $\mathcal A_K\subseteq \mathcal A_L \cap K$.
For the other inclusion, let $a\in \mathcal A_L\cap K$. Let $v\in L$ and $n\in \mathbb N$ be such that $\delta_L(v)=nav$. We write $v=fg^{-1}$, where $f=\sum_{i=0}^{\ell}f_i w^i \in K[w]$ and $g=\sum_{j=0}^{m}g_j w^j \in K[w]$ are polynomials with $f_{\ell} \neq 0$ and $g_m \neq 0$.
From $\delta_L(v)=nav$, we have
$\delta_L(f)g-f\delta_L(g)=nafg$.
Substituting $f=\sum_{i=0}^{\ell}f_i w^i, ~g=\sum_{j=0}^{m}g_j w^j$
and comparing coefficients of $w^{l+m}$, we get
$\delta(f_{\ell})g_m-f_{\ell}\delta(g_m)=naf_{\ell}g_m$.
Consequently, $\delta(f_{\ell}g_m^{-1})=na (f_{\ell}g_m^{-1})$.
Thus $f_{\ell}g_m^{-1}\in K$ is a solution of equation $\delta(y)=nay$, hence $a\in \mathcal A_K$.
\end{proof}

The next corollary is a counterpart of Corollary \ref{trdeg-and-rank-AS-prime}.

\begin{corollary}\label{trdeg-and-AS-prime-and-derivative-1}
Let $(K, \delta)$ be a differential field such that $\delta(w)\neq 1$ for every $w \in K$. Let $S \subseteq K\backslash\{0\}$
be a finite set. Let $(E, \delta_E)$ be a differential extension of $(K, \delta)$ such that for each $a \in S \backslash\{0\}$, the field $E$ contains solutions of the set of equations
\begin{equation}
\delta(y_1) = ay_1, \delta(y_2)=y_1+ay_2.
\end{equation}
Then ${\rm trdeg}_K(E) \geq {\rm rank}(\mathcal A_S') + 1$.
\end{corollary}
\begin{proof}
Since $(E, \delta)$ is a splitting field of $S$, for each $a \in S$ the field $E$ contains solutions of the equations $\delta(y_2) = y_1 + ay_2$. Thus, by Lemma \ref{derivative_1}, there exists a transcendental element $w \in E$ such that $\delta_E(w) = 1$. Let $L = K(w)$ and $\delta_L$ be the restriction of $\delta_E$ to $L$.
Since $K$ does not contain an element whose derivative is $1$, by Lemma \ref{A_K-and-A_K(w)}, $\mathcal A=\mathcal A_K=\mathcal A_L \cap K$. Taking intersection with $S$ and noting that $S \subseteq K$, we obtain $\mathcal A_S = (\mathcal A_L)_S$. Therefore
${\rm rank} (\mathcal A_S') = {\rm rank}((\mathcal A_L)'_S)$.
Now, by Corollary \ref{trdeg-and-rank-AS-prime}, ${\rm trdeg}_L(E) \geq {\rm rank}((\mathcal A_L)_S') = {\rm rank}(\mathcal A_S')$. 
Thus, ${\rm trdeg}_K(E) = {\rm trdeg}_L(E) + {\rm trdeg}_K(L) \geq {\rm rank}(\mathcal A_S') + {\rm trdeg}_K(L) = {\rm rank}(\mathcal A_S') + 1$.
\end{proof}

\section{Finite Splitting Fields of Differential Central Simple Algebras}\label{section-algebraic-extensions-splitting-dcsa}
Throughout this section, $K$ denotes the field $F(t)$, where $F$ is a field of characteristic $0$ and
$\delta$ denotes the derivation on $K$ determined by $\delta(t)=t$ and $\delta(a) = 0$ for each $a \in F$.
In this section, we discuss the splitting of
differential matrix algebras $(M_n(K), \mathcal{D}_P)$ over $(K, \delta)$, where $P$ is a traceless matrix over $K$ and $\mathcal{D}_P = \delta^c + \partial_P$ as in \S\ref{section-preliminaries}. First, we shall look at the case when $P \in M_n(F)$, and later consider other cases. Following lemma will be useful in proofs of main theorems.

\begin{lemma}\label{F-intersection-A-is-Q}
Let  
$\mathcal A = \{a \in K : \delta(y)=nay \text{ has a nonzero solution in } K \text{ for some } n\in \mathbb N\}$. 
Then $F \cap \mathcal A = \mathbb Q$. 
\end{lemma}
\begin{proof}
Let $a \in \mathbb Q$. Since $\delta(t) = t$ and $\mathcal A$ is a 
$\mathbb Q$-vector space, $\mathbb Q \subseteq \mathcal A$.
Since ${\rm char}(F) = 0$, $\mathbb Q \subseteq F \cap \mathcal A$.

Now, suppose $a \in F \cap \mathcal A$. Then there exist $f, g\in F[t]$; $g \neq 0$, such that $\delta(f/g)=naf/g$ for some $n \in \mathbb N$. We rewrite it as
$$\frac{\delta(f)g-f\delta(g)}{g^2} = \frac{naf}{g}.$$

Let $f = a_{\ell}t^{\ell}+\dots+a_1t+a_0$ and $g=b_mt^m+\dots+b_1t+b_0$, where $a_i, b_j\in F$. Then $\delta(f) = {\ell}a_{\ell}t^{\ell}+\dots+a_1t$ and $\delta(g) = mb_mt^m+\dots+b_1t$.
Substituting it in the above equation,
\begin{equation*}
\begin{split}
({\ell}a_{\ell}t^{\ell}+\dots+a_1t)(b_mt^m+\dots+b_1t+b_0)-(a_{\ell}t^{\ell}+\dots+a_1t+a_0)(mb_mt^m+\dots+b_1t)\\
=na (a_{\ell}t^{\ell}+\dots+a_1t+a_0)(b_mt^m+\dots+b_1t+b_0).
\end{split}
\end{equation*}
Let $i$ be the smallest index such that $a_i\neq 0$ and $j$ be the smallest index such that $b_j \neq 0$. Comparing coefficients of $t^{i+j}$,
$$ia_ib_j-ja_ib_j=na a_ib_j.$$
Thus $a = (i-j)/n \in \mathbb Q$.
\end{proof}

\begin{theorem}\label{split-by-finite}
Let $P\in M_n(F)$ be a traceless matrix. The differential matrix algebra $(M_n(K) , \mathcal{D}_P)$ admits a finite splitting field if and only if $P$ is diagonalizable and every eigenvalue of $P$ is rational.
\end{theorem}

\begin{proof}
We first assume that $P$ is diagonalizable, and every eigenvalue of $P$ is in $\mathbb Q$. 
Let $D = {\rm diag}(\lambda_1, \lambda_2, \cdots, \lambda_n)$ be 
a diagonal matrix with eigenvalues of $P$, with multiplicity, as diagonal entries. 
Let $S=\{\lambda_1,\dots,\lambda_n\}$ and $$\mathcal A=\{a \in K : \delta(y) = nay \text{ has a nonzero solution in } K \text{ for some } n \in \mathbb N\}.$$
Since $\delta(t) = t$, by Lemma \ref{F-intersection-A-is-Q} $S \subseteq \mathcal A$. Hence $\mathcal{A}_S'=0$. We recall that $\mathcal A_S'$ is determined by the direct sum
$\mathcal A \cap \langle S \rangle + \mathcal A_S' = \langle S \rangle$.
Thus by Theorem \ref{transcendental-iff-AS-prime-nonzero}, there exists a splitting field $E$ of $S$ such that $E/K$ is finite.
Thus $\delta(y) = \lambda_iy$ has a solution $u_i\in E \setminus \{0\}$, for each $i$. Consequently, the equation $\delta^c(Y)=D Y$ has a solution ${\rm diag}(u_1,u_2,\dots, u_n) \in {\rm GL}_n(E)$, where
$Y=(y_{ij})$ is an $n \times n$ matrix of indeterminates. Now, from Lemma \ref{similar-matrices}, the equation $\delta^c(Y)=P^tY$  has a solution
in ${\rm GL}_n(E)$. By Corollary \ref{trivialize}, the extension $E/K$ splits 
$(M_n(K), \mathcal D_P)$.

Conversely, suppose that $P$ is not diagonalizable or there is an eigenvalue of $P$ that is not in $\mathbb{Q}$. We consider two cases.

\noindent \textbf{Case 1 :} \emph{Each eigenvalue $\lambda_i$ of $P$ is rational.}

Let $(E,\delta_E)$ be a finite differential extension of $(K,\delta)$.
Let $J = {\rm diag}(J_1, J_2, \cdots, J_k)$ be a Jordan canonical form of $P$, where $J_i$ is an upper triangular Jordan block.
As $P$ is not diagonalizable, at least one $J_i$ has its length strictly greater than $1$. We assume that the length of $J_1$ is 
$\ell > 1$ and that the eigenvalue that corresponds to $J_1$ is $\lambda_1$.

Since $E/K$ is finite and there does not exist an element $v\in K$ such that $\delta(v)=1$, by Corollary \ref{derivative_1_in_alg_ext} there does not exist $w \in E$ such that $\delta_E(w) = 1$. By Lemma \ref{derivative_1}, the equations $\delta(y_1)=\lambda_1 y_1$ and $\delta(y_2)=y_1+\lambda_1y_2$ do not have solutions in $E$.

Thus the equation $\delta^c(Y)=J^t Y$ has no solution over $E$, where
$Y=(y_{ij})$ is an $n \times n$ matrix of indeterminates. Now, from Lemma \ref{similar-matrices}, the equation $\delta^c(Y)=P^tY$  has no solution
over $E$. By Corollary \ref{trivialize}, the extension $E$ does not split 
$(M_n(K), \mathcal D_P)$. 

\noindent \textbf{Case 2 :} \emph{At least one eigenvalue of $P$ is not rational}. 

Let $\overline F$ be an algebraic closure of $F$. Let $S \subseteq \overline{F}$ be the set of eigenvalues of $P$. Suppose $(E,\delta_E)$ is a finite differential extension of $(\overline F(t),\delta_{\overline F(t)})$.
Let
$$\mathcal A=\{a \in \overline F(t) : \delta(y) = nay \text{ has a nonzero solution in } \overline F(t) \text{ for some } n \in \mathbb N\}.$$
By Lemma \ref{F-intersection-A-is-Q}, $\overline F \cap \mathcal A = \mathbb Q$. Thus, if $\lambda$ is a nonrational eigenvalue of $P$, then $\lambda \notin \mathcal A$. Therefore, $\mathcal A_S' \neq 0$. 
Since $E$ is a finite extension over $\overline{F}(t)$,
by Theorem \ref{transcendental-iff-AS-prime-nonzero}, the equation $\delta(y) = \lambda y$ does not have a solution in $E$.
Consequently, the equation $\delta^c(Y)=J^t Y$ has no solution over $E$, where $J$ is a Jordan canonical form of $P$ and $Y=(y_{ij})$ is an $n \times n$ matrix of indeterminates. Now, from Lemma \ref{similar-matrices}, the equation $\delta^c(Y)=P^tY$  has no solution over  $E$.
Therefore, by Corollary \ref{trivialize}, $(E, \delta_E)$ does not split $(M_n(K), \mathcal D_P)$.

Now, if $(L, \delta_L)$ is a finite extension of $(K,\delta)$ that
splits $(M_n(K), \mathcal D_P)$, then the compositum of $(L, \delta_L)$
and $(\overline{F}(t),\delta_{\overline{F}(t)})$ is a finite extension of $(\overline{F}(t),\delta_{\overline{F}(t)})$ splitting 
$(M_n(K), \mathcal D_P)$, which is a contradiction.
Thus a finite extension of $(K,\delta)$ does not split $(M_n(K), \mathcal D_P)$.
\end{proof}

\begin{corollary}\label{deg_E}
Let $P\in M_n(F)$ be a traceless matrix. If the differential matrix algebra $(M_n(K) , \mathcal{D}_P)$ admits a finite splitting field $(E, \delta_E)$ over $(K, \delta)$, then 
${\rm deg}_K(E) \geq [\langle S \cup \{1\} \rangle:\mathbb Z]$, where $S \subseteq \mathbb Q$ is the set of eigenvalues of $P$.
\end{corollary}

\begin{proof}
Let $S \subseteq \overline{F}$ be the set of eigenvalues of $P$.
Since $(M_n(K), \mathcal{D}_P)$ admits a finite splitting field,
by Theorem \ref{split-by-finite},  $S \subseteq \mathbb Q$. Thus, $\langle S \rangle \subseteq \mathbb Q$ is a cyclic group.
Let $q \in \mathbb N$ be such that $1/q$ generates $\langle S \rangle$. If $(E,\delta_E)$ is a finite extension of $(K, \delta)$ that splits $(M_n(K), \mathcal{D}_P)$, the equation $\delta(y) = (1/q)y$ has a solution $z \in E$. Thus $\delta_E(z^q)=z^q$ and consequently, $z^q=ct$ for some $c\in C_E$. Since $E/K$ is algebraic, by \cite[Lemma 3.19]{Magid-Book}, $c$ is algebraic over $C_K = F$.

Let $I$ be the ideal of $F(c)[t]$ generated by $t$.
Then $ct \in I$ but $ct \notin I^2$ and by generalized Eisenstein criterion, the polynomial $X^q - ct \in K(c)[X]$ is irreducible. 
Therefore, ${\rm deg}_K(E) \geq {\rm deg}_{K(c)}(E) \geq q = [\langle S \cup \{1\}\rangle:\mathbb Z]$.
\end{proof}

Let us examine $n = 2$ case. If $P \in M_2(F)$ is a traceless matrix, then $(M_2(K), \mathcal D_P)$ is split by a finite extension of $(K, \delta)$ if and only if $-\det(P)$ is a square in $\mathbb Q^*$. Further, in that case, if $-\det(P)=p^2/q^2$ is in reduced fractional form, then there exists a finite extension of degree $q$ that splits $(M_2(K), \mathcal D_P)$. We observe that $q = [\langle S \cup \{1\}\rangle : \mathbb Z]$, where $S = \{-p/q, p/q\}$. More generally, in the above corollary the lower bound $[\langle S \cup \{1\}\rangle : \mathbb Z]$ on ${\rm deg}_K(E)$ is sharp, which is attained when $c \in K$ and $E = K(z)$. The following corollary can be derived from $q = 1$ case.

\begin{corollary}\label{integer eigenvalues}
Let $P\in M_n(F)$ be a traceless matrix and $(A, \mathcal{D}_P)$ be the
differential matrix algebra of degree $n$ over $K$. The following are equivalent.
\begin{enumerate}
\item[(i).] $(A, \mathcal{D}_P)$ is split.
\item[(ii).] $(A, \mathcal{D}_P)$ is split over 
$(\overline{F}(t), \delta_{\overline{F}(t)})$.
\item[(iii).] $P$ is diagonalizable and all the eigenvalues of $P$ are in $\mathbb{{Z}}$. 
\end{enumerate}
\end{corollary}
\subsection{Finite splitting of $(A,\mathcal{D}_P)$ when $\delta^c(P)\neq 0$}

In this section also, we assume the field $(K,\delta)$ to be $(F(t),\delta)$ such that $\delta(a)=0$ for every $a\in F$ and $\delta(t)=t$. Let $(M_n(K),\mathcal{D}_P)$ be the differential matrix algebra, where $\delta^c(P)\neq 0$, i.e., 
at least one entry of $P$ is a nonconstant element in $F(t)$.

In the following theorem, we find some conditions on $P$ which ensure that $(M_n(K), \mathcal D_P)$ is split by a finite extension.

\begin{theorem}\label{splitting-of-non-constant-matrices}
Let $P\in M_n(K)$ be a traceless matrix.
Then $(M_n(K) , \mathcal{D}_P)$ is split by a finite extension if $P$ is of one of the following forms.

\begin{enumerate}
\item [(i).] $P={\rm diag}(d_1,d_2,\dots,d_n)$ and each
$d_i \in \mathcal A$.
\item [(ii).] 
$P = (p_{ij})$, $p_{ij}\in F[t]$ is an upper triangular matrix satisfying 
$p_{ii}\in \mathbb Q$ and $p_{ii}\mathbb Z\neq p_{jj}\mathbb Z\in \mathbb Q/\mathbb Z$, unless $i = j$.
\end{enumerate}
\end{theorem}
\begin{proof}
$(i).$ Since $P = \diag(d_1,d_2,\dots,d_n)$, the set of eigenvalues of $P$ is $S = \{d_1,d_2,\dots,d_n\}$, which, by hypothesis, is contained in $\mathcal A$. Thus $\mathcal A_S' = 0$ and from Theorem \ref{transcendental-iff-AS-prime-nonzero}, there exists a finite extension $E$ over $K$ such that the equation $\delta(y)=d_iy$ has a solution $u_i\in E \setminus \{0\}$. Thus the equation $\delta^c(Y)=P^t Y$ has a solution ${\rm diag}(u_1,u_2,\dots, u_n) \in {\rm GL}_n(E)$, where
$Y=(y_{ij})$ is an $n \times n$ matrix of indeterminates. By Corollary \ref{trivialize}, the extension $E/K$ splits 
$(M_n(K), \mathcal D_P)$. \\

$(ii).$ Since $P$ is upper triangular, $S := \{p_{11}, p_{22}, \cdots, p_{nn}\}$ is the set of eigenvalues of $P$. By hypothesis, $S \subseteq \mathbb Q$ and by Lemma \ref{F-intersection-A-is-Q}, $\mathbb Q= F \cap \mathcal A$. Thus, $S \subseteq \mathcal A$ and $\mathcal A_S' = 0$.
Now by Theorem \ref{transcendental-iff-AS-prime-nonzero}, there exists a finite extension $E/K$
and $u_i \in E \setminus \{0\}$ such that $\delta_E(u_i) = p_{ii}u_i$. We show that $E$ splits $(M_n(K), \mathcal D_P)$. By Corollary \ref{trivialize}, it suffices to show that the equation $\delta^c(Y)=P^tY$ has a solution
in ${\rm GL}_n(E)$. It is enough to find a lower triangular matrix with nonzero diagonal entries satisfying $\delta^c(Y)=P^tY$. Let $Y=(y_{ij})$, where $y_{ij}$ are indeterminates, and write the $j^{\rm th}$ column of the above system of equations
\begin{equation}\label{eqnn}
\delta(y_{ij})=\sum_{k=j}^{i}p_{ki}y_{kj}
\end{equation}
for $j \leq i \leq n$.
We construct a lower triangular matrix $Z = (z_{ij})$, with $z_{jj} = u_j \in E\setminus \{0\}$ as a solution of this system. 	

For each $j \leq n$, we use induction on $i$, with $j \leq i \leq n$, to show that
there exist $r_{ij} \in C_E[t]$ such that the elements $z_{jj}, \cdots, z_{ij}, \cdots, z_{nj} \in E$ defined by $z_{ij} := r_{ij} u_j\in E$,
are solutions of the system of equations in \ref{eqnn}.
For $i=j$, this is evident with $r_{jj}=1$ and $z_{jj}=u_j$, since the corresponding equation in the system \ref{eqnn} is $\delta(y_{jj})=p_{jj}y_{jj}$.	
By induction hypothesis, assume that for $j\leq i\leq n-1$, there exist $r_{ij}\in C_E[t]$ such that $z_{ij} := r_{ij} u_j\in E$ satisfies first $n-j$ equations of \ref{eqnn}. The last equation is 
$$\delta(y_{nj})=\sum_{k=j}^{n}p_{kn}y_{kj}.$$

Substituting $y_{kj} = z_{kj} := r_{kj} u_j$, whenever $j\leq k\leq n-1$, we get 
$$\delta(y_{nj})-p_{nn} y_{nj} = \sum_{k=j}^{n-1} p_{kn} r_{kj} u_j.$$

We multiply this equation by $u_n^{-1}$ and use 
$\delta_E(u_n^{-1})=-p_{nn} u_n^{-1}$, to obtain

\begin{equation}\label{eqnn_2}
\delta(y_{nj} u_n^{-1})=\delta(y_{nj})u_n^{-1}-p_{nn}y_{nj}u_n^{-1}=\left(\sum_{k=j}^{n-1} p_{kn} r_{kj}\right) u_ju_n^{-1}.
\end{equation}

We expand the polynomial
$\sum_{k=j}^{n-1} p_{kn} r_{kj} \in C_E[t]$ in powers of $t$ and write
$$\sum_{k=j}^{n-1} p_{kn} r_{kj} = 
\sum_{l=1}^{m}a_{njl}t^l,$$
where $a_{njl}\in C_E$. Thus
\begin{equation}\label{eqnn_3}
\delta(y_{nj} u_n^{-1})=\sum_{l=1}^{m}a_{njl}t^l u_ju_n^{-1}.
\end{equation}
Since for each $j\leq n-1$, $p_{nn}\neq p_{jj}$ modulo $\mathbb Z$, the rational number $l+p_{jj}-p_{nn}$ is nonzero for each $l \in \mathbb Z$. We compute

$$
\delta_E((l+p_{jj}-p_{nn})^{-1} a_{njl} t^l u_j u_n^{-1})
= a_{njl} t^l u_j u_n^{-1}.
$$

Making this substitution in \ref{eqnn_3},
$$
\delta(y_{nj} u_n^{-1})=\delta_E\left(\sum_{l=1}^{m}
(l+p_{jj}-p_{nn})^{-1} a_{njl}t^l u_ju_n^{-1}\right).
$$
We define $r_{nj} :=\sum_{l=1}^{m}
(l+p_{jj}-p_{nn})^{-1} a_{njl}t^l \in C_E[t]$ and rewrite the above equation as 
$$\delta(y_{nj} u_n^{-1})=\delta_E\left(r_{nj} u_ju_n^{-1}\right).
$$
Thus $z_{nj} := r_{nj} u_j\neq 0$ is a solution of \ref{eqnn_3}.

Finally, we vary $j$ to get $z_{ij} := r_{ij} u_j \in E$, for each $1 \leq j \leq i \leq n$, with $r_{jj} = 1$.
such that the lower triangular matrix $Z:=(z_{ij})\in M_n(E)$ satisfies $\delta^c(Y)=p^tY$. Since each $z_{jj} = u_j$ is nonzero, $Z \in {\rm GL}_n(E)$. 
\end{proof}

\subsection{Minimal transcendence degree of a splitting field.} \label{section-transcendental-extensions-splitting-dcsa}

In this subsection, we are concerned about finding transcendence degrees of splitting fields of differential matrix algebras which do not admit a finite splitting field. 

In the following theorem, we assume our field to be $K=F(t)$ with derivation $\delta$ such that $\delta(a)=0$ for each $a\in F$ and $\delta(t)=t$.

\begin{theorem}\label{tr_splitting_field}
Let $0 \neq P\in M_n(F)$ be a traceless matrix. Let $S \subseteq \overline{F}$ be the set of eigenvalues of $P$ and $k = {\rm dim}_{\mathbb Q}({\rm span}_{\mathbb Q}	(S\cup \{1\} ))$.
Let $(E,\delta_E)$  be a differential field extension of $(K, \delta)$ that splits the differential matrix algebra $(M_n(K), \mathcal{D}_P)$. 
\begin{enumerate}
\item[(i).] If $P$ is diagonalizable over $\overline{F}$, then ${\rm trdeg}_K(E) \geq k-1$.
\item[(ii).] If $P$ is not diagonalizable over $\overline{F}$, then ${\rm trdeg}_K(E) \geq k$.
\end{enumerate}
\end{theorem}
\begin{proof} Since 
${\rm trdeg}_K(E) = {\rm trdeg}_{\overline{F}(t)}(E\overline{F})$, where $E\overline{F}$ is the subfield of $\overline{E}$ generated by $E$ and $\overline{F}$, we may assume that $F$ is algebraically closed.

\noindent $(i).$ Since $P$ is diagonalizable over $\overline{F}  =F$ and $(E, \delta_E)$ splits $(M_n(K), \mathcal D_P)$, as in the proof of Theorem \ref{split-by-finite}, for each $\lambda \in S$, the field $E$ contains solutions of $\delta(y)=\lambda y$. 
Further, $\mathcal A_S = \langle S \rangle \cap \mathcal A \subseteq F \cap \mathcal A = \mathbb Q$, where the last equality holds by Lemma \ref{F-intersection-A-is-Q}. Thus by Corollary \ref{trdeg-and-rank-AS-prime}, ${\rm trdeg}_K(E)\geq {\rm rank}(\mathcal A_S') = {\rm dim}_{\mathbb Q}({\rm span}_{\mathbb Q} (S\cup \{1\} ))-1 = k-1$.

$(ii).$ Suppose that $P$ is not diagonalizable over $\overline{F} = F$. Then, as in the proof of Theorem \ref{split-by-finite}, the field $E$ contains a solution of the set of equations $\delta(y)=\lambda y$ for all $\lambda\in S$ and for some $a\in S$, $E$ contains solution of the system $\delta(y)=ay$, $\delta(y_1)=y+a y_1$. Thus by Corollary \ref{trdeg-and-AS-prime-and-derivative-1}, ${\rm trdeg}_K(E) \geq 1 + {\rm rank}(\mathcal A_S') = {\rm dim}_{\mathbb Q}({\rm span}_{\mathbb Q} (S\cup \{1\} ))=k$.
\end{proof}


From the above theorem, the only case where $(M_n(K), {\mathcal D}_P)$ admits a finite splitting field is when $P$ is diagonalizable 
over $\overline{F}$ and $k = 1$. However, if we extend the zero derivation of $F$ to $K$,
by defining $\delta_{c,m}(t) = ct^m$, where $m \neq 1$ is an integer and $0\neq c\in F$, then such a possibility is not realized.
In the following theorem we find sharp bounds on transcendence degrees of the extensions that split differential matrix 
algebra $(M_n(K),\mathcal{D}_P)$ over $(F(t), \delta_{c,m})$, where $P \in M_n(F)$ is a traceless matrix.

\begin{theorem}\label{tr_splitting_field_ct-m}
Let $K = F(t)$, $0 \neq c \in F$ and $m \neq 1$ be an integer. Let $\delta_{c,m}$ be the derivation on $K$ defined by $\delta_{c,m}(a)=0$, if $a\in F$, and $\delta_{c,m}(t)=ct^m$. Let $0\neq P\in M_n(F)$ be a traceless matrix and $(E,\delta_E)$ be a differential field extension of $(K, \delta_{c,m})$ that splits $(M_n(K), \mathcal{D}_P)$. Then $E$ is not algebraic over $K$. In fact, ${\rm trdeg}_{K}(E) \geq {\rm rank}(\langle S \rangle)$, where
$S \subseteq \overline{F}$ is the set of eigenvalues of $P$.
\end{theorem}

\begin{proof}
Let.
$$\mathcal A=\{a \in \overline F(t) : \delta_{c,m}(y) = nay \text{ has a nonzero solution in } \overline F(t) \text{ for some } n \in \mathbb N\}.$$
We claim that $\mathcal A_S := \mathcal A \cap S = 0$. Since $\mathcal A = {\rm span}_{\mathbb Q}({\mathcal L}_{\delta_{c,m}}(\overline{F}(t)))$ and
$S \subseteq \overline{F}$, it is enough to show that $\delta_{c,m}(a)a^{-1} \notin \overline{F}$ for every
$a\in \overline{F}(t)\setminus \overline{F}$.
Let $a=fg^{-1}\in \overline{F}(t)\setminus \overline{F}$, where $f,g \in \overline{F}[t]$
are polynomials of degrees $r$ and $s$, respectively, and ${\rm gcd}(f,g)=1$. Then
$$\delta_{c,m}(a)=\frac{\delta_{c,m}(f)g-f\delta_{c,m}(g)}{g^2}=\frac{(f'g-fg')ct^m}{g^2},$$
where $f'$ and $g'$ denote the images of the derivation defined by $t' = 1$.
If $\delta_{c,m}(a)a^{-1} = k \in \overline{F}$, then
the above equation may be rewritten as
\begin{equation}\label{eq_1}
(f'g-fg')ct^m=kfg.
\end{equation}
We compare the degrees on both sides of the equation \ref{eq_1} and obtain
$$r+s+m-1 \geq r+s.$$
If $m\leq 0$, the above inequality cannot hold. Thus $\delta_{c,m}(a)a^{-1} \notin \overline{F}$.

Consider the case where $m > 1$.
Let $\ell$ be the highest power of $t$ that divides $f$. Since $\gcd(f,g) = 1$ and $t^{\ell - 1}$ divides $f'g - fg'$, comparing highest powers of $t$ dividing the two sides of the equation \ref{eq_1}, $\ell - 1 + m \leq \ell$. This is a contradiction, since $m > 1$. Thus
$\delta_{c,m}(a)a^{-1} \notin \overline{F}$.

We conclude that $\mathcal A_S = 0$ in either case. Consequently, $\mathcal A_S' = \langle S \rangle$. We now proceed to the proof of the theorem. Let $(E, \delta_E)$ be a differential extension of $(K, \delta_{c,m})$ that splits $(M_n(K), \mathcal D_P)$.

First, assume that $P$ is diagonalizable. 
By Corollary \ref{trivialize}, $E$ contains solutions of differential equations $\delta_{c,m}(y)=a y$ for each $a \in S$. Now by 
Corollary \ref{trdeg-and-rank-AS-prime}, ${\rm trdeg}_{K\overline{F}}(E\overline{F}) \geq {\rm rank}(\mathcal A_S')= {\rm rank}(\langle S\rangle)$. 
As $\overline{F}(t)/K$ is algebraic, ${\rm trdeg}_{K}(E) = {\rm trdeg}_{\overline{F}(t)}(E\overline{F})\geq {\rm rank}(\langle S\rangle)$.

Now assume that $P$ is not diagonalizable. 
By Corollary \ref{trivialize}, $E$ contains solutions of equations $\delta_{c,m}(y)=ay$ and $\delta_{c,m}(y_1)=y+ay_1$ for some $a\in S$. Observe that $w := (c(m-1))^{-1}t^{1-m} \in K$ is such that $\delta(w)=1$. We now use Lemma \ref{derivative_1}, Corollary \ref{trdeg-and-rank-AS-prime} and argue as in the case when $P$ is diagonalizable to show that ${\rm trdeg}_{K}(E) \geq {\rm rank}(\mathcal A_S')= {\rm rank}(\langle S\rangle)$. 
\end{proof}

We assert that the bounds on ${\rm trdeg}_K(E)$ in Theorem \ref{tr_splitting_field} and Theorem \ref{tr_splitting_field_ct-m} are sharp. Using Proposition \ref{deg_tr}, we may construct a field $E$ whose transcendence degree is equal to these bounds. Using the solutions of $\delta(y) = \lambda y$; $\lambda \in S$ it is easy to construct a block diagonal matrix $Z \in {\rm GL}_n(E)$ consisting of lower triangular invertible blocks such that $\delta^c(Z) = J_P^tZ$, where $J_P$ is a Jordan form of $P$.
This, in conjunction with Corollary \ref{trivialize} and Lemma \ref{similar-matrices}, proves the assertion.

\section{Splitting of tensor powers}\label{section-splitting-of-tensor-powers}
Let $K = F(t)$ and $\delta$ the derivation on $K$ defined by $\delta(a)=0$ for $a\in F$ and $\delta(t)=t$. In this section, we are concerned about the splitting of tensor powers of a matrix differential algebra $(M_n(K),\mathcal D_P)$. We also relate it to the existence of finite splitting fields of $(M_n(K),\mathcal D_P)$.

\begin{theorem}\label{no-algebraic-then-infinite-order}
Let $P \in M_n(F)$ be a traceless matrix.
Let $(M_n(K),\mathcal{D}_P)$ be a differential matrix algebra over $K$. 
If $(M_n(K),\mathcal{D}_P)$ does not admit a finite splitting field over $(K, \delta)$, then the order of $(M_n(K),\mathcal{D}_P)$ is infinite.
\end{theorem}
\begin{proof}
Let $m$ be a positive integer such that 
$(A,\mathcal{D}_Q) := (M_n(K),\mathcal{D}_P)^{\otimes m}$ is split over $(K, \delta)$. By Theorem \ref{split-by-finite} and Corollary \ref{integer eigenvalues}, $Q$ is diagonalizable, and all its eigenvalues are integers. Hence by Lemma \ref{m_fold}, $P$ is diagonalizable, and all the eigenvalues of $P$ are rational.
By Theorem \ref{split-by-finite}, $(M_n(K),\mathcal{D}_P)$ is split by a finite extension over $(K, \delta)$, which is a contradiction to the hypothesis.
\end{proof}

To construct an example of a differential matrix algebra of finite order, let us consider the product 
$(M_2(K), \mathcal D_{P_1}) \otimes
(M_2(K), \mathcal D_{P_2})$, where $P_1,P_2 \in M_2(F)$ are traceless diagonalizable matrices with rational eigenvalues $\pm \lambda_1$ and $\pm \lambda_2$, respectively. Consider the differential matrix algebra $(M_2(K),\mathcal D_{P_1})\otimes (M_2(K),\mathcal D_{P_2})$. Derivation on this algebra is $\mathcal D_{P_1}\otimes 1+1 \otimes \mathcal D_{P_2}=\mathcal D_{P_1\otimes 1+1\otimes P_2}$. From Lemma \ref{tensor_n},
eigenvalues of $P_1\otimes 1+1\otimes P_2$ are $\pm \lambda_1 \pm \lambda_2$.
Thus by Theorem \ref{split-by-finite} and Corollary \ref{deg_E}, there exists
a finite extension $E/K$ of degree $d$ that splits $(M_2(K),\mathcal D_{P_1})\otimes (M_2(K),\mathcal D_{P_2})$, where $d$ is given by
$$d= 
\begin{cases}
{\rm lcm}(q_1,q_2)/2 & \text{ if ${\rm lcm}(q_1, q_2)$ is even and $p_1,p_2$ are both odd}\\ 
{\rm lcm}(q_1,q_2) & \text{ otherwise},
\end{cases}
$$
where $\lambda_i = p_i/q_i$ and $p_i, q_i$ are relatively prime integers.
Thus, if each $\lambda_i = \frac{1}{2}$ then $(M_2(K),\mathcal D_{P_1})\otimes (M_2(K),\mathcal D_{P_2})$ is split over $(K,\delta)$ itself.

In the next theorem, we use these ideas to deal with tensor powers of a differential matrix algebra $(M_n(K),\mathcal{D}_P)$. We find conditions on the eigenvalues of $P$ so that the $m$-fold tensor product of $(M_n(K),\mathcal{D}_P)$ is trivial.

\begin{theorem}\label{m-fold_tensor}
Let $P\in M_n(F)$ be a traceless matrix such that the differential matrix algebra
$(M_n(K), \mathcal{D}_P)$  is split by a finite extension. Then the following are equivalent:
\begin{enumerate}
\item[(i).] $(A,\mathcal{D}_Q):=(M_n(K), \mathcal{D}_P)^{\otimes m}$ is split over $(K, \delta)$.
\item[(ii).] If $\lambda$ and $\mu$ are any two eigenvalues of $P$, then $m\{\lambda\} = m\{\mu\}$, where $\{\lambda\}$ and $\{\mu\}$ denote fractional parts of $\lambda$ and $\mu$, respectively.
\end{enumerate}
We note that, since $(M_n(K),\mathcal{D}_P)$ is split by a finite extension, eigenvalues of $P$ are rational and it makes sense to talk about fractional parts.
\end{theorem}
\begin{proof}
Let $S := \{\lambda_1,\lambda_2,\dots,\lambda_n\}$ be the set of eigenvalues of $P$. 
Since $(M_n(K), \mathcal D_P)$ is split by a finite extension, by Theorem \ref{split-by-finite}, the matrix $P$ is diagonalizable and $S \subseteq \mathbb Q$.
Since $\mathcal D_Q=\mathcal{D}_P^{\otimes m}$,
by Lemma \ref{tensor_n}, the set of eigenvalues of $Q$ is 
$$\Lambda := \{\lambda_{i_1}+\lambda_{i_2}+\dots +\lambda_{i_m}\,:\,(i_1, i_2, \cdots, i_m) \in \{1,2,\dots,n\}^m\}.$$ 

\emph{(i)}$\Rightarrow$\emph{(ii)}: Since $(A,\mathcal{D}_Q)$ is split over $(K, \delta)$, by Corollary \ref{integer eigenvalues} each $\mu \in \Lambda$ is an integer. Thus, from the particular case $i_1=i_2=\dots =i_m$, each $m\lambda_i\in \mathbb{Z}$. Let $\alpha_i=m\lambda_i$. 
Now, from the case $i_1 = \dots = i_{m-1}\neq i_m$, we get that
$a_{ij} := (m-1)\lambda_i+\lambda_j \in \mathbb Z$, whenever $i \neq j$. Thus $\lambda_j-\lambda_i=a_{ij}-\alpha_i \in \mathbb{{Z}}$, and therefore, $\{\lambda_i\}$ is the same for each $i$.

\emph{(ii)}$\Rightarrow$\emph{(i)}: Let $m\{\lambda_i\}=a \in \mathbb{{Z}}$ for each $i$. Then each eigenvalue
$\lambda_{i_1}+\lambda_{i_2}+\dots +\lambda_{i_m}$ of $Q$ is an integer and by Corollary \ref{integer eigenvalues},
$(A, \mathcal D_Q)$ is split over $(K, \delta)$. 
\end{proof}

From Theorems \ref{no-algebraic-then-infinite-order} and Theorem \ref{m-fold_tensor}, we conclude that if $(M_n(K), \mathcal D_P)$ 
has finite order, then this order is precisely the smallest integer $m$ such that $m\{\lambda_i\} \in \mathbb N$, where $\{\lambda_i\}$ is the constant fractional part of eigenvalues of $P$. In the following corollary we show that the 
order of $(M_n(K),\mathcal{D}_P)$ is necessarily a divisor of $n$.

\begin{corollary}\label{order-divides-degree}
Let $P\in M_n(F)$ be a traceless matrix. If the differential matrix algebra $(M_n(K), \mathcal{D}_P)$ has finite order $m$, then $m$ divides $n$.
\end{corollary}

\begin{proof}
Let $S := \{\lambda_1,\lambda_2,\dots,\lambda_n\}$, be the set of eigenvalues of $P$. Let $a_{ij}, \alpha_i \in \mathbb Z$ be as in the proof of 
Theorem \ref{m-fold_tensor}. Then $(m-1)\alpha_i+\alpha_j=m a_{ij}$. Since $P$ is a traceless matrix, $0 = m\left(\sum_{j=1}^{n}\lambda_j\right) = \sum_{j=1}^{n}\alpha_j$. Substituting $\alpha_j = m(a_{1j}-\alpha_1)+\alpha_1$ whenever $j \neq 1$,
$$\sum_{j=1}^{n}\alpha_j = \alpha_1 + \sum_{j=2}^{n}(\alpha_1 + m(a_{1j}-\alpha_1)) = n\alpha_1+mc=0,$$ 

where $c=\sum_{j=2}^{n}(a_{1j}-\alpha_1) \in \mathbb Z$. Now 
$\alpha_1 = m\lambda_1 = m([\lambda_1] + \{\lambda_1\}) = m[\lambda_1] + a$, where 
$a$ is as in the proof of Theorem \ref{m-fold_tensor} and $[\lambda_1]$ is the integer part of $\lambda_1$. It is now clear that
$\gcd(\alpha_1,m)=\gcd(a,m)=1$ and $n\alpha_1+mc = 0$. Thus $m$ divides $n$.
\end{proof}
We conclude this article by asking the following intriguing question. \emph{"Let $(A, \mathcal D)$ be a differential central simple algebra over a differential field $(K, \delta)$. If $(A, \mathcal D)^{\otimes m}$ is split for some $m \in \mathbb N$, then does $m$ necessarily divide ${\rm deg}(A)$?"}

\bibliographystyle{amsalpha}
\bibliography{dcsa}
\end{document}